\documentclass[12pt]{article}
\usepackage{amsmath}
\usepackage{amssymb}
\usepackage{amsthm}
\usepackage{verbatim}
\usepackage{mathrsfs}

\newcommand{\R}[0]{\mathbb R}

\newcommand{\Ds}[0]{\mathcal D}

\newtheorem{Th}{Theorem}[section]
\newtheorem{Lemma}{Lemma}[section]
\newtheorem{Prop}[Lemma]{Proposition}

\begin{document}

\title{On the well-posedness of the hyperelastic rod equation}
\author{H. Inci}

\maketitle

\begin{abstract}
In this paper we consider the hyperelastic rod equation on the Sobolev spaces $H^s(\R)$, $s > 3/2$. Using a geometric approach we show that for any $T > 0$ the corresponding solution map, $u(0) \mapsto u(T)$, is nowhere locally uniformly continuous. The method applies also to the periodic case $H^s(\mathbb T)$, $s > 3/2$.
\end{abstract}

\section{Introduction}\label{section_introduction}

We consider the following family of equations referred to as the hyperelastic rod equation
\begin{equation}\label{hr}
u_t-u_{txx} + 3uu_x = \gamma ( 2 u_x u_{xx} + uu_{xxx}),\quad t \in \R, x \in \R
\end{equation}
where $\gamma \neq 0$. The initial value problem for \eqref{hr} is locally well-posed in the Sobolev spaces $H^s, s > 3/2$ -- see \cite{yin,zhou}. For the corresponding solution map it was shown in \cite{olson} in the periodic case that it has not the property to be uniformly continuous on bounded sets, whereas in \cite{karapetyan} the same was shown for both cases (periodic and non-periodic) with an improvement for the $s$ range. Our aim here is to prove that the solution map for the range $s > 3/2$ has even less regularity. More precisely

\begin{Th}\label{th_nonuniform}
Let $s > 3/2$ and $T > 0$. Denote by $\Phi_T$ the time $T$ solution map of the initial value problem for \eqref{hr} defined on $U_T \subseteq H^s(\R)$. Then
\[
 \Phi_T:U_T \to H^s(\R),\quad u(0) \mapsto u(T)
\]
is nowhere locally uniformly continuous.
\end{Th}   

We will rewrite \eqref{hr} by doing the transformation $v(t,x)=u(t,\gamma x)$. This gives
\[
 v_t - \frac{1}{\gamma^2} v_{txx} + \frac{3}{\gamma} v v_x = \frac{1}{\gamma^2} ( 2v_x v_{xx} + v v_{xxx}) 
\]
or rewritten
\[
 (1-\frac{1}{\gamma^2} \partial_x^2) (v_t + vv_x) = \frac{\gamma-3}{\gamma} v v_x - \frac{1}{\gamma^2} v_x v_{xx}
\]
and equivalently
\begin{equation}\label{hr2}
  v_t + vv_x = (1-\frac{1}{\gamma^2} \partial_x^2)^{-1}(\frac{\gamma-3}{\gamma} v v_x - \frac{1}{\gamma^2} v_x v_{xx})=:B(v,v)
\end{equation}
Note that $B$ is a continuous quadratic form on $H^s(\R)$ for $s > 3/2$. It will be enough to establisch Theorem \ref{th_nonuniform} for $v$ instead of $u$. As will be shown in the next section \eqref{hr2} is convenient for the geometric framework.

\section{The Geometric framework}\label{section_framework}

We will formulate \eqref{hr2} in a geometric way as was done in \cite{b_family}. Consider the flow map of $v$, i.e. 
\[
 \varphi_t(t,x) = v(t,\varphi(t,x)),\quad \varphi(0,x)=x
\]
The functional space for the $\varphi$ variable is for $s > 3/2$ the diffeomorphism group
\[
 \Ds^s(\R)=\{ \varphi:\R \to \R \;|\; \varphi-\mbox{id} \in H^s(\R),\quad \varphi_x(x) > 0 \mbox{ for all } x \in \R \}
\]
where $\mbox{id}$ is the identity map in $\R$. It is a topological group and consits of $C^1$-diffeomorphisms. For details on this space -- see \cite{composition}. We can write \eqref{hr2} in the $\varphi$ variable as
\begin{equation}\label{secondorder}
 \varphi_{tt} = B(\varphi_t \circ \varphi^{-1},\varphi_t \circ \varphi^{-1}) \circ \varphi
\end{equation}
The computations in \cite{b_family} show that right side is a real analytic map
\[
 \Ds^s(\R) \to P_2(H^s(\R);H^s(\R)), \quad \varphi \mapsto [v \mapsto B(v \circ \varphi^{-1},v \circ \varphi^{-1}) \circ \varphi]
\]
where we denote by $P_2(H^s(\R);H^s(\R))$ the space of continuous quadratic forms on $H^s(\R)$ with values in $H^s(\R)$. We can write the second order equation \eqref{secondorder} as a first order equation on the tangent space $T\Ds^s(\R)=\Ds^s(\R) \times H^s(\R)$
\begin{equation}\label{firstorder}
 \partial_t \left( \begin{array}{c} \varphi \\ v \end{array} \right) = \left(\begin{array}{c} v \\ B(v \circ \varphi^{-1},v \circ \varphi^{-1}) \circ \varphi \end{array}\right)
\end{equation}
The quadratic nature of the second component makes it to a so called Spray -- see \cite{lang}. It has in particular an exponential map. To define this map consider the ODE \eqref{firstorder} with initial values $\varphi(0)=\mbox{id}$ and $v(0)=v_0$. Denote by $V \subseteq H^s(\R)$ those initial values $v_0$ for which we have existence beyond time $1$. With this we define
\[
 \exp:V \subseteq H^s(\R) \to \Ds^s(\R),\quad v_0 \mapsto \varphi(1;v_0)
\]
where $\varphi(1;v_0)$ is the time $1$ value of the $\varphi$-component. Because of analytic dependence on initial values $\exp$ is real analytic. Furthermore for any $v_0 \in H^s(\R)$ the curve $\varphi(t)=\exp(t v_0)$ is the $\varphi$-component of the solution to \eqref{firstorder} with initial values $\varphi(0)=\mbox{id}$ and $v(0)=v_0$. In particular the solution exists as long as $tv_0 \in V$.\\
With this we can construct solutions to \eqref{hr2}. So consider \eqref{hr2} with initial condition $v(0)=v_0 \in H^s(\R)$. For $\varphi(t)=\exp(tv_0)$ we define 
\[
 v(t)=\varphi_t(t) \circ \varphi(t)^{-1}
\]
It turns out that $v$ is a solution to \eqref{hr2} -- see \cite{b_family}. By the local well-posedness for ODEs we immediately recover local well-posedness -- see \cite{yin,zhou}.

\begin{Th}
The initial value problem for \eqref{hr} is locally well-posed in the Sobolev spaces $H^s(\R), s > 3/2$.
\end{Th}

\section{Nonuniform dependence}\label{section_nonuniform}

In this section we establish our main result Theorem \ref{th_nonuniform}. As mentioned already it will be enough to prove this for the modified equation \eqref{hr2}. We can further simplify by considering the theorem by just considering the time $T=1$ situation as we have for $v$ a solution to \eqref{hr2} that
\[
 \tilde v(t,x):=\lambda v(\lambda t,x)
\] 
is also a solution to \eqref{hr2}.\\
We proceed as in \cite{b_family}. In \cite{b_family} we used a conserved quantity to establish the result. For equation \eqref{hr2} we have something similar.

\begin{Lemma}\label{lemma_conserved}
Let $s >3/2$. For $v$ a solution to \eqref{hr2} with initial value $v(0)=v_0 \in H^s(\R)$ we have
\begin{equation}\label{conservation}
 \left((1-\frac{1}{\gamma^2}\partial_x^2) v(t)\right) \circ \varphi(t) \cdot \varphi_x(t)^2 = (1-\frac{1}{\gamma^2}\partial_x^2) v_0 + \int_0^t \frac{3\gamma-3}{\gamma} \frac{\varphi_t(s) \varphi_{tx}(s)}{\varphi_x(s)} \;ds
\end{equation}
where $\varphi(t)=\exp(tv_0)$.
\end{Lemma}

The essential thing here is that the ''remainder'' term, the integral term, is more regular then the first term. 

\begin{proof}
We differentiate the expression on the left in the Lemma with respect to $t$ and use the equation for $v$. 
We have by the chain rule
\begin{eqnarray*}
 \frac{d}{dt} (((1-\frac{1}{\gamma^2}\partial_x^2)v)\circ \varphi)&=&((1-\frac{1}{\gamma^2}\partial_x^2)v_t)\circ \varphi + ((1-\frac{1}{\gamma^2}\partial_x^2)v_x)\circ \varphi \cdot \varphi_t\\
&=&\left((1-\frac{1}{\gamma^2}\partial_x^2)v_t + ((1-\frac{1}{\gamma^2}\partial_x^2)v_x) \cdot v\right) \circ \varphi\\
&=& \left((1-\frac{1}{\gamma^2}\partial_x^2)v_t + (1-\frac{1}{\gamma^2}\partial_x^2)(v_x \cdot v) + \frac{3}{\gamma^2} v_x v_{xx}\right) \circ \varphi\\
&=& (\frac{\gamma-3}{\gamma} vv_x + \frac{2}{\gamma^2} v_x v_{xx}) \circ \varphi 
\end{eqnarray*}
where we used equation \eqref{hr2} in the last equality. Therefore we have
\begin{eqnarray*}
\frac{d}{dt} (((1-\frac{1}{\gamma^2}\partial_x^2)v)\circ \varphi \cdot \varphi_x^2) &=& (vv_x + \frac{2}{\gamma^2} v_x v_{xx}) \circ \varphi \cdot \varphi_x^2 +(v-\frac{1}{\gamma^2} v_{xx}) 2 \varphi_x \varphi_{tx} \\ &=& \frac{3\gamma-3}{\gamma} (v v_x) \circ \varphi = \frac{3\gamma-3}{\gamma} \frac{\varphi_t \varphi_{tx}}{\varphi_x}
\end{eqnarray*}
where we used $\varphi_{tx} \varphi_x^{-1} = v_x \circ \varphi$. As $\varphi(0)=\mbox{id}$ integrating gives the result in the case where we work with regular solutions. But as long as $||v_x||_{L^\infty}$ is controlled (similar to the Beale-Majda-Kato criterium) one has continuation of the solution -- see \cite{zhou}. Thus by approximation by regular solutions one has \eqref{conservation} for all $s > 3/2$.
\end{proof}

\noindent
In the following we will use the notation 
\[
 y(t):=(1-\frac{1}{\gamma^2} \partial_x^2) v(t) \mbox{ and } \Psi(t)=\int_0^t \frac{3\gamma-3}{\gamma} \frac{\varphi_t(s) \varphi_{tx}(s)}{\varphi_x(s)} \;ds
\]
Hence from \eqref{conservation}
\[
 y(1) = \left(\frac{y(0)}{\varphi_x(1)^2}\right) \circ \varphi(1)^{-1} + \left(\frac{\Psi(1)}{\varphi_x(1)^2} \right) \circ \varphi(1)^{-1}  
\]
In the following we will use also $\Psi_{v_0}:=\Psi(1)$ for the corresponding initial value $v_0$. Theorem \eqref{th_nonuniform} will follow from

\begin{Prop}\label{prop_nonuniform}
Let $V \subseteq H^s(\R)$ be the domain of definition of $\exp$. We denote by $v(t)$ solutions to \eqref{hr2}. Then the map 
\[
 \Phi:V \subseteq H^s(\R) \to H^s(\R),\quad v(0) \mapsto v(1)
\]
is nowhere locally uniformly continuous.
\end{Prop}

To prove Proposition \ref{prop_nonuniform} we will show that $y(0) \mapsto y(1)$ is nowhere locally uniformly continuous. This is clearly enough. Before doing this we state some facts -- see \cite{b_family} for the proofs.\\
For $\varphi_\bullet \in \Ds^s(\R)$ there is $C > 0$ with
\[
 \frac{1}{C} ||\left(\frac{y}{\widetilde \varphi_x^2}\right) \circ \varphi^{-1}||_{s-2} \leq ||y||_{s-2} \leq C ||\left(\frac{y}{\widetilde \varphi_x^2}\right) \circ \varphi^{-1}||_{s-2}
\]
for all $y \in H^{s-2}(\R)$ and for all $\widetilde \varphi, \varphi$ in some neighborhood of $\varphi_\bullet$.\\
For $\varphi_\bullet \in \Ds^s(\R)$ there is $C > 0$ with
\[
 ||f \circ \varphi_1^{-1} - f \circ \varphi_2^{-1}||_{s-2} \leq C ||f||_{s-1} ||\varphi_1^{-1}-\varphi_2^{-1}||_{s-2} 
\]
for all $f \in H^{s-1}(\R)$ and for all $\varphi_1,\varphi_2$ in a neighborhood of $\varphi_\bullet$.\\
Further we construct a dense subset $S \subseteq V$ with $S \subseteq H^{s+1}(\R)$ and $d_v \exp \neq 0$ for all $v \in S$. Here $d_v \exp$ is the differential. Take an arbitrary $v \in V \cap H^{s+1}(\R)$ and $w \in H^s(\R), x \in \R$ with $w(x) \neq 0$. Consider the analytic map
\[
 \R \to \R,\quad t \mapsto \left(d_{tv}\exp(w)\right)(x)
\]
which at $t=0$ is $w(x)$, in particular nonzero. Thus there is a sequence $t_n \uparrow 1$ with $\left(d_{t_nv}\exp(w)\right)(x) \neq 0$. So putting $t_n v$ to $S$ gives the construction we need.\\
With this preparation we can proceed to the proof of Proposition \ref{prop_nonuniform}. It is essentially the same proof as in \cite{b_family}.

\begin{proof}[Proof of Proposition \ref{prop_nonuniform}]
We take $v_0 \in S \subseteq H^{s+1}(\R)$ in the dense subspace and show that $\Phi$ is not uniformly continuous on any ball $B_R(v_0) \subseteq V$ of radius $R >0$ with center $v_0$. By the construction of $S$ we fix $g \in H^s(\R)$ and $x_0 \in \R$ with 
\[
 \left(d_{v_0} \exp(g)\right)(x_0) > m ||g||_s
\] 
for some $m > 0$. Denote $\varphi_\bullet=\exp(v_0)$. We choose $R_1 > 0$ in such a way that we have
\[
 \frac{1}{C_1} ||\left(\frac{y}{\widetilde \varphi_x^2}\right) \circ \varphi^{-1}||_{s-2} \leq ||y||_{s-2} \leq C_1 ||\left(\frac{y}{\widetilde \varphi_x^2}\right) \circ \varphi^{-1}||_{s-2}
\]
for some $C_1 > 0$ for all $y \in H^{s-2}(\R)$ and $\tilde \varphi, \varphi \in \exp(B_{R_1}(v_0))$. Taking $0 < R_2 \leq R_1$ we can garantuee that
\[
 ||y \circ \varphi^{-1}||_{s-2} \leq C_2 ||y||_{s-2}
\]
for some $C_2$ and for all $y \in H^{s-2}(\R)$ and $\varphi \in \exp(B_{R_2}(v_0))$. Choosing $0 < R_3 \leq R_2$ we can ensure 
\[
 ||f \circ \varphi_1^{-1}-f \circ \varphi_2^{-1}||_{s-2} \leq \tilde C_3 ||f||_{s-1} ||\varphi_1^{-1}-\varphi_2^{-1}||_{s-2} \leq C_3 ||f||_{s-1} ||\varphi_1-\varphi_2||_s
\]
for some $C_3 > 0$ and for all $f \in H^{s-1}(\R)$ and $\varphi_1,\varphi_2 \in \exp(B_{R_3}(v_0))$. Furthermore we denote by $C > 0$ the constant in the Sobolev imbedding
\[
 ||f||_{L^\infty} \leq C ||f||_s
\]
Consider the Taylor expansion for $\exp$
\[
 \exp(w+h) = \exp(w) + d_w \exp(h) + \int_0^1 (1-t) d_{w+th}\exp(h,h) \;dt
\]
We choose $0 < R_4 \leq R_3$ in such a way that we have
\[
 ||d_w^2 \exp(h_1,h_2)||_s \leq K ||h_1||_s ||h_2||_s
\]
and
\[
 ||d_{w_1}^2 \exp(h_1,h_2)-d_{w_2}^2 \exp(h_1,h_2)||_s \leq K ||w_1-w_2||_s ||h_1||_s ||h_2||_s
\]
for some $K > 0$ and for all $w,w_1,w_2 \in \exp(B_{R_4}(v_0))$ and for all $h_1,h_2 \in H^s(\R)$. By taking $0 < R_5 \leq R_4$ small enough we have
\[
 \max\{C \cdot K \cdot R_5,C \cdot K \cdot R_5^2\} < m/2
\]
By the final choice $0 < R_\ast \leq R_5$ we can make
\[
 |\varphi(x)-\varphi(y)| \leq L |x-y| \mbox{ and } ||\Psi_v||_s \leq M \mbox{ and } ||\exp(v)-\exp(\tilde v)||_s \leq L ||v-\tilde v||_s
\]
to hold for all $\varphi \in \exp(B_{R_\ast})$ and $v,\tilde v \in B_{R_\ast}(v_0)$. The goal is to prove that $\Phi$ is not uniformly continuous on $B_R(v_0)$ for any $0 < R \leq R_\ast$. So we fix $0 < R \leq R_\ast$. We define the sequence of radii \[
 r_n = \frac{m}{8n} ||g||_s, \quad n \geq 1
\]
and take arbitrary smooth $w_n$ with support in $(x_0-\frac{r_n}{L},x_0+\frac{r_n}{L})$ and constant mass $||w_n||_s = R/4$. Further we define $g_n=g/n$, which tends to zero in $H^s(\R)$. With this we introduce two sequences
\[
 z_n=v_0 + w_n \quad \mbox{and} \quad \tilde z_n= z_n + g_n=v_0 + w_n + g_n
\]
For $N$ large enough we clearly have $z_n, \tilde z_n \in B_R(v_0)$ for $n \geq N$ and $||z_n-\tilde z_n||_s \to 0$ as $n \to \infty$. Further we introduce the corresponding diffeomorphisms
\[
 \varphi_n=\exp(z_n) \quad \mbox{and} \quad \widetilde \varphi_n=\exp(\tilde z_n)
\]
The result will follow from $\limsup_{n \to \infty} ||\Phi(z_n)-\Phi(\tilde z_n)||_s > 0$. Reexpressing $\Phi$ with \eqref{conservation} and using the notation $y_n=(1-\frac{1}{\gamma^2}\partial_x^2)z_n$ and $\tilde y_n=(1-\frac{1}{\gamma^2}\partial_x^2)\tilde z_n$ and $\Psi_{z_n},\Psi_{\tilde z_n}$ for the ''remainder'' terms this is equivalent to
\[
 \limsup_{n \to \infty} ||\frac{y_n}{(\varphi_n)_x^2} \circ \varphi_n^{-1} - \frac{\tilde y_n}{(\widetilde \varphi_n)_x^2} \circ \widetilde \varphi^{-1} + \frac{\Psi_{z_n}}{(\varphi_n)_x^2} \circ \varphi_n^{-1} - \frac{\Psi_{\tilde z_n}}{(\widetilde \varphi_n)_x^2} \circ \widetilde \varphi_n^{-1}||_{s-2} > 0
\]
As the $\Psi$ terms are more regular than $H^{s-2}$, namely in $H^{s-1}$, we have
\begin{multline*}
 ||\frac{\Psi_{z_n}}{(\varphi_n)_x^2} \circ \varphi_n^{-1} - \frac{\Psi_{\tilde z_n}}{(\widetilde \varphi_n)_x^2} \circ \widetilde \varphi_n^{-1}||_{s-2} \\
 \leq ||\frac{\Psi_{z_n}}{(\varphi_n)_x^2} \circ \varphi_n^{-1} - \frac{\Psi_{z_n}}{(\varphi_n)_x^2} \circ \widetilde \varphi_n^{-1}||_{s-2} + ||\frac{\Psi_{z_n}}{(\varphi_n)_x^2} \circ \widetilde \varphi_n^{-1} - \frac{\Psi_{\tilde z_n}}{(\widetilde \varphi_n)_x^2} \circ \widetilde \varphi_n^{-1}||_{s-2} \\
\leq C_3 ||\frac{\Psi_{z_n}}{(\varphi_n)_x^2} ||_{s-1} ||\varphi_n-\widetilde \varphi_n||_s + C_2 ||\frac{\Psi_{z_n}}{(\varphi_n)_x^2}- \frac{\Psi_{\tilde z_n}}{(\widetilde \varphi_n)_x^2}||_{s-2} \to 0
\end{multline*}
as $n \to \infty$ since $z \mapsto \Psi_z/(\partial_x \exp(z))^2$ is a smooth. Thus it remains to establish
\[
 \limsup_{n \to \infty} ||\frac{y_n}{(\varphi_n)_x^2} \circ \varphi_n^{-1} - \frac{\tilde y_n}{(\widetilde \varphi_n)_x^2} \circ \widetilde \varphi_n^{-1}||_{s-2} > 0
\]
We split 
\[
 y_n = (1-\frac{1}{\gamma^2}\partial_x^2) (v_0 + w_n) \mbox{ resp. } \tilde y_n = (1-\frac{1}{\gamma^2}\partial_x^2)(v_0 + w_n + g_n)
\]
As $v_0 \in H^{s+1}$ we can treat the $v_0$ terms in the same way as the 
$\Psi$ terms and get
\[
 \lim_{n \to \infty} ||\frac{(1-\frac{1}{\gamma^2}\partial_x^2)v_0}{(\varphi_n)_x^2} \circ \varphi_n^{-1} - \frac{(1-\frac{1}{\gamma^2}\partial_x^2)v_0}{(\widetilde \varphi_n)_x^2} \circ \widetilde \varphi_n^{-1}||_{s-2} = 0
\]
For the $g_n$ term we have trivially
\[
 ||\frac{(1-\frac{1}{\gamma^2}\partial_x^2)g_n}{(\widetilde \varphi_n)_x^2} \circ \widetilde \varphi_n^{-1}||_{s-2} \leq C_1 ||(1-\frac{1}{\gamma^2}\partial_x^2)g_n||_{s-2} \to 0
\]
The only remaining thing is to consider
\[
 \limsup_{n \to \infty} ||\frac{(1-\frac{1}{\gamma^2}\partial_x^2)w_n}{(\varphi_n)_x^2} \circ \varphi_n^{-1} - \frac{(1-\frac{1}{\gamma^2}\partial_x^2)w_n}{(\widetilde \varphi_n)_x^2} \circ \widetilde \varphi_n^{-1}||_{s-2} 
\]
In order to estimate this from below we will establish that the two terms have disjoint support. This we do by estimating the distance $|\varphi_n(x_0)-\widetilde \varphi_n(x_0)|$. By the Taylor expansion we have
\[
 \varphi_n = \exp(v_0) + d_{v_0} \exp(w_n) + \int_0^1 (1-t) d_{v_0+tw_n}^2 \exp(w_n,w_n) \;dt
\]
resp.
\[
 \widetilde \varphi_n = \exp(v_0) + d_{v_0} \exp(w_n+g_n) + \int_0^1 (1-t) d_{v_0 + t(w_n+g_n)}^2 \exp(w_n+g_n,w_n+g_n) \;dt
\]
Taking the difference we can write
\[
 \widetilde \varphi_n - \varphi_n = d_{v_0}\exp(g_n) + \mathcal R_1 + \mathcal R_2 + \mathcal R_3
\]
where 
\[
 \mathcal R_1 = \int_0^1 (1-t) (d_{v_0+t(w_n+g_n)}^2(w_n,w_n)-d_{v_0+tw_n}^2(w_n,w_n))\;dt
\]
and
\[
 \mathcal R_2 = \int_0^1 (1-t) d_{v_0+t(w_n+g_n)}^2(g_n,g_n) \;dt
\]
and
\[
 \mathcal R_2 = 2 \int_0^1 (1-t) d_{v_0+t(w_n+g_n)}^2(w_n,g_n) \;dt
\]
For these we have
\[
 ||\mathcal R_1||_\infty \leq C ||\mathcal R_1||_s \leq C K ||g_n||_s ||w_n||_s^2 \leq \frac{1}{n} C K ||g||_s (R/4)^2 \leq \frac{1}{4n} C K R^2 ||g||_s 
\]
and 
\[
 ||\mathcal R_2||_\infty \leq C ||\mathcal R_2||_s \leq 2 C K ||g_n||_s ||w_n||_s \leq \frac{1}{n} C K ||g||_s (R/4) \leq \frac{2}{4n} C K R ||g||_s
\]
and
\[
 ||\mathcal R_3||_\infty \leq C ||\mathcal R_3||_s \leq  C K ||g_n||_s^2 \leq \frac{1}{n} C K ||g||_s (R/4) \leq \frac{1}{4n} C K R ||g||_s
\]
Therefore
\begin{eqnarray*}
 |\varphi_n(x_0)-\widetilde \varphi_n(x_0)| &\geq &|d_{v_0}\exp(g_n)| - ||\mathcal R_1||_\infty - ||\mathcal R_2||_\infty - ||\mathcal R_3||_\infty \\
&\geq& \frac{1}{n} m ||g||_s - \frac{1}{n} \frac{m}{2} ||g||_s = \frac{m}{2n} ||g||_s
\end{eqnarray*}
The support of $\frac{(1-\frac{1}{\gamma^2}\partial_x^2)w_n}{(\varphi_n)_x^2} \circ \varphi_n^{-1}$ is contained in $(\varphi_n(x_0)-r_n,\varphi_n(x_0)+r_n)$ taking into account the lipschitz property of $\varphi_n$ with lipschitz constant $L$ and the definition of $w_n$. Analogously the support of $\frac{(1-\frac{1}{\gamma^2}\partial_x^2)w_n}{(\widetilde \varphi_n)_x^2} \circ \widetilde \varphi_n^{-1}$ is contained in $(\widetilde \varphi_n(x_0)-r_n,\widetilde \varphi_n(x_0)+r_n)$. So we have
\[
 r_n \leq |\varphi_n(x_0)-\tilde \varphi_n(x_0)|/4
\]
This allows us (see \cite{b_family}) to ''separate'' the disjointly supported terms. Thus we have
\begin{multline*}
 \limsup_{n \to \infty} ||\frac{(1-\frac{1}{\gamma^2}\partial_x^2)w_n}{(\varphi_n)_x^2} \circ \varphi_n^{-1}-\frac{(1-\frac{1}{\gamma^2}\partial_x^2)w_n}{(\widetilde \varphi_n)_x^2} \circ \widetilde \varphi_n^{-1}||_{s-2}^2 \\
\geq \limsup_{n \to \infty} \tilde C (||\frac{(1-\frac{1}{\gamma^2}\partial_x^2)w_n}{(\varphi_n)_x^2} \circ \varphi_n^{-1}||_{s-2}^2+||\frac{(1-\frac{1}{\gamma^2}\partial_x^2)w_n}{(\widetilde \varphi_n)_x^2} \circ \widetilde \varphi_n^{-1}||_{s-2}^2)\\
\geq \limsup_{n \to \infty} \tilde C \frac{2}{C^2} ||(1-\frac{1}{\gamma^2}\partial_x^2)w_n||_{s-2}^2 \geq \limsup_{n \to \infty} \tilde K ||w_n||_s^2 = \tilde K R^2/4
\end{multline*}
So for any $R \leq R_\ast$ we have constructed $(z_n)_{n \geq N},(\tilde z_n)_{n \geq N} \subseteq B_R(u_0)$ with $\lim_{n \to \infty} ||z_n-\tilde z_n||_s=0$ and $\limsup_{n \to \infty} ||\Phi(z_n)-\Phi(\tilde z_n)||_s \geq C \cdot R$ for some constant $C > 0$ independent of $R$ showing the claim.

\end{proof}

\bibliographystyle{plain}

\flushleft
\author{ Hasan Inci\\
EPFL SB MATHAA PDE \\
MA C1 627 (B\^atiment MA)\\ 
Station 8 \\
CH-1015 Lausanne\\
Schwitzerland\\
        {\it email: } {hasan.inci@epfl.ch}
}

\end{document}